\documentclass[reqno,A4paper]{amsart}

\usepackage{amsmath}
\usepackage{amssymb}
\usepackage{amsthm}
\usepackage{enumerate}
\usepackage{mathrsfs} 
\usepackage{eqlist}
\usepackage{array}

\theoremstyle{plain}
\newtheorem{theorem}{Theorem}[section]

\newtheorem{lemma}{Lemma}[section]
\newtheorem{corol}{Corollary}[theorem]

\theoremstyle{definition}

\newtheorem{remark}{\textup{Remark}} 
\newtheorem{example}{\textit{Example}} 

\numberwithin{equation}{section}

\begin{document}
	
	\title[$*$-Ricci solitons and gradient almost $*$-Ricci solitons]%
	{$*$-Ricci solitons and gradient almost $*$-Ricci solitons on Kenmotsu manifolds}
	\author[Venkatesha \and Devaraja Mallesha Naik \and
	H Aruna Kumara]%
	{Venkatesha* \and Devaraja Mallesha Naik** \and
		H Aruna Kumara***}
	
	\newcommand{\acr}{\newline\indent}
	
	\address{\llap{*\,}Department of Mathematics\acr
		Kuvempu University\acr
		Shivamogga, Karnataka\acr
		 INDIA}
	\email{vensmath@gmail,com}
	
	\address{\llap{**\,}Department of Mathematics\acr
		Kuvempu University\acr
		Shivamogga, Karnataka\acr
		INDIA}
	\email{devarajamaths@gmail.com}
	
	\address{\llap{**\,}Department of Mathematics\acr
		Kuvempu University\acr
		Shivamogga, Karnataka\acr
		INDIA}
	\email{arunmathsku@gmail.com}

	
	\thanks{The second author (D.M.N.) is grateful to University Grants Commission, New Delhi (Ref. No.:20/12/2015(ii)EU-V) for financial support in the form of Junior Research Fellowship.}
	
	
	\begin{abstract}
	In this paper, we consider $*$-Ricci soliton in the frame-work of  Kenmotsu manifolds. First, we prove that if the metric of a Kenmotsu manifold $M$ is a $*$-Ricci soliton, then soliton constant $\lambda$ is zero. For 3-dimensional case, if $M$ admits a $*$-Ricci soliton, then we show that $M$ is of constant sectional curvature $-1$. Next, we show that if $M$ admits a $*$-Ricci soliton whose potential vector
	field is collinear with the characteristic vector field $\xi$, then $M$ is Einstein and soliton vector field is equal to $\xi$. Finally, we prove that if $g$ is a gradient almost $*$-Ricci soliton, then either $M$ is Einstein or the potential vector field is collinear with the characteristic vector field on an open set  of $M$. We verify our result by constructing examples for both $*$-Ricci soliton and gradient almost $*$-Ricci soliton.
	\end{abstract}
	
	\maketitle
	
\section{Introduction}
A \textit{Ricci soliton} on a Riemannian manifold $(M,g)$ is defined by 
\begin{equation}\label{E:00}
\pounds_V g+2 S+2\lambda g=0,
\end{equation}
where $\pounds$ denotes the Lie derivative operator, $\lambda$ is a constant and $S$ is the Ricci tensor of the metric $g$. Ricci soliton is a natural generalization of the Einstein metric (that is, $S=ag$, for some constant $a$), and  is a special self similar solution of the Hamilton's Ricci flow (see \cite{Ham}) $\frac{\partial}{\partial t}g(t)=-2S(t)$ with initial condition $g(0)=g$. We say that the Ricci soliton is \textit{shrinking} when $\lambda<0$, \textit{steady} when $\lambda=0$, and \textit{expanding} when $\lambda >0$. If the vector field $V$ is the gradient of a smooth function $f$ (denoted by $Df$, where $D$ indicates the gradient operator), then $g$ is called a \textit{gradient Ricci soliton} and in such a case \eqref{E:00} becomes
\begin{align}\label{Eq:01}
\text{Hess } f+S+\lambda g=0,
\end{align}
where $\text{Hess } f$ is the Hessian of the smooth function $f$. Equations \eqref{E:00} and \eqref{Eq:01} are respectively called \textit{almost Ricci soliton} and \textit{gradient almost Ricci soliton}, if $\lambda$ is a variable smooth function on $M$. We recommend the reference \cite{Chow} for more details about the Ricci flow and Ricci soliton. In this connection, we mention that within the framework of contact geometry Ricci solitons were first considered by Sharma in \cite{Sharma}.

The notion of $*$-Ricci tensor was first introduced by Tachibana \cite{Tachibana} on almost Hermitian manifolds and Hamada \cite{Hamada} apply this notion of $*$-Ricci tensor to almost contact manifolds defined by
\begin{equation*}
S^*(X,Y)=\frac{1}{2}trace(Z\to R(X, \varphi Y)\varphi Z),
\end{equation*}
for any  $X,Y\in \mathfrak{X}(M)$ (where $\mathfrak{X}(M)$ is the Lie algebra of all vector fields on $M$). In 2014, Kaimakamis and Panagiotidou \cite{Kaimakamis} introduced the concept of $*$-Ricci solitons within the framework of real hypersurfaces of a complex space form, where they essentially modified the definition of Ricci soliton by replacing the Ricci tensor $S$ in \eqref{E:00} with the $*$-Ricci tensor $S^*$. More precisely, a Riemannian metric $g$ on a manifold $M$ is called a $*$-Ricci soliton if there exists a constant $\lambda$ and a vector field $V$ such that
\begin{align}\label{Eq:03}
\pounds_V g+2S^*+2\lambda g=0.
\end{align}
Moreover, if the vector field $V$ is a gradient of a smooth function $f$, then we say that it is gradient $*$-Ricci soliton and \eqref{Eq:03} becomes
\begin{align}\label{Eq:04}
\text{Hess } f+S^*+\lambda g=0.
\end{align}
Note that  $*$-Ricci soliton is trivial if the vector field $V$ is Killing, and in this case the manifold becomes $*$-Einstein (that is, $S^*=\alpha g$, for some function $\alpha$).  If $\lambda$ appearing in \eqref{Eq:03} and \eqref{Eq:04} is a variable smooth function on $M$, then $g$ is called \textit{almost $*$-Ricci soliton} and \textit{gradient almost $*$-Ricci soliton} respectively.

Very recently in 2018, Ghosh and Patra \cite{Ghosh3} first undertook the study of $*$-Ricci solitons on almost contact metric manifolds. In their paper, the authors proved that if the metric of Sasakian manifold is a $*$-Ricci Soliton, then it is either positive Sasakian, or null-Sasakian. Furthermore, they also proved that if a complete Sasakian metric is a gradient almost $*$-Ricci Soliton, then it is positive-Sasakian and isometric to a unit sphere $S^{2n+1}$. Here we also mention the works of Prakasha and Veeresha \cite{Prakasha} within the frame-work of paracontact geometry. This motivates the present authors to consider Kenmotsu manifolds whose metric as a $*$-Ricci soliton and gradient almost $*$-Ricci soliton. 

The present paper is organized as follows: In Section~\ref{S:02}, we recall some fundamental definitions related to Kenmotsu manifolds. Section~\ref{S:03} is devoted to the study of Kenmotsu manifolds whose metric is a $*$-Ricci soliton. We prove that if a Kenmotsu manifold $M$ admits a $*$-Ricci soliton, then the soliton constant $\lambda=0$. Moreover, we show that if $M$ is an $\eta$-Einstein manifold of dimension $>3$ admitting $*$-Ricci soliton, then $M$ is Einstein.  For Kenmotsu 3-manifold admitting $*$-Ricci soliton, we prove that $M$ is of constant negative curvature $-1$. At the end of this section, we give an example of Kenmotsu manifold which admits a $*$-Ricci soliton, and verify our results. The final section deals with Kenmotsu manifold $M$ admitting gradient almost $*$-Ricci soliton, and we show that in such a case either $M$ is Einstein or the soliton vector field is collinear with characteristic vector field on an open set of $M$. The section ends with an example of a gradient almost $*$-Ricci soliton on a  Kenmotsu 3-manifold.

\section{Preliminaries}\label{S:02}
A $(2n+1)$-dimensional smooth manifold $M$ is said to
have an \textit{almost contact structure} if it admits a tensor field $\varphi$ of type $(1, 1)$, a vector field $\xi$ (called the \textit{characteristic vector field} or \textit{Reeb vector field}), and a 1-form
$\eta$ such that
\begin{equation} \label{E:02}
\varphi^2=-I+\eta\otimes\xi,\quad \eta (\xi)=1.
\end{equation}
An immediate consequence of \eqref{E:02} is that $\varphi\xi=0$ and $\eta\circ\varphi=0$. It is well known that a $(2n+1)$ dimensional  smooth manifold $M$ admits an almost contact structure if and only if the structure group of the tangent bundle of $M$ reduces to $U(n)\times 1$. For more details, we refer to \cite{Blair}.

If $M$ with $(\varphi, \xi, \eta)$-structure  admits a Riemannian metric $g$ such that
$g(\varphi X, \varphi Y)=g( X, Y)-\eta(X)\eta(Y),$
for all $X,Y\in \mathfrak{X}(M)$, then $(M, \varphi, \xi, \eta, g)$ is called an \textit{almost contact metric manifold}. We consider the sign convension of the Riemannian curvature tensor as $R(X, Y)=[\nabla_X, \nabla_Y]-\nabla_{[X, Y]}$.

An almost contact metric manifold $(M, \varphi, \xi, \eta, g)$ is said to be  \textit{Kenmotsu} (see \cite{Kenm}) if 
\begin{equation}\label{EQ:06}
(\nabla_X \varphi)Y=g(\varphi X, Y)\xi -\eta(Y)\varphi X,
\end{equation} 
for any $X,Y\in \mathfrak{X}(M)$. For a Kenmotsu manifold, we also have (see \cite{Kenm})
\begin{gather}
\nabla_X \xi =X-\eta(X)\xi,\label{E:04}\\
R(X, Y)\xi=\eta(X)Y-\eta(Y)X,\label{E:05}\\
S(X,\xi)=-2n\eta(X),\label{E:06}\\
(\pounds_\xi g)(X,Y)=2\{g(X,Y)-\eta(X)\eta(Y)\},\label{E:007}
\end{gather}
for all $X,Y\in \mathfrak{X}(M)$.
Note that \eqref{E:007} implies that $\xi$ is not Killing in Kenmotsu manifold. We say $M$ is \textit{$\eta$-Einstein} if the Ricci tensor satisfy
\begin{equation}\label{E:08}
S=\alpha g+\beta\eta\otimes\eta,
\end{equation}
for certain smooth function $\alpha$ and $\beta$. If $\beta=0$, then $M$ becomes an \textit{Einstein manifold}. From \eqref{E:08} and \eqref{E:06}, we have
\begin{equation}\label{E:4.2}
\alpha+\beta=-2n.
\end{equation}
Contracting \eqref{E:08} and using \eqref{E:4.2}, we get
\begin{equation*}
\alpha=\left( \frac{r}{2n}+1\right) , \quad \beta=-\left( \frac{r}{2n}+2n+1\right).
\end{equation*}
Thus, a Kenmotsu manifold $M$ is $\eta$-Einstein if and only if 
\begin{equation}\label{E:010}
S=\left( \frac{r}{2n}+1\right)g-\left( \frac{ r}{2n}+2n+1\right)\eta\otimes\eta.
\end{equation}

\section{*-Ricci soliton on Kenmotsu manifolds}\label{S:03}
First we need the following lemmas.
\begin{lemma}
	A $(2n+1)$ dimensional Kenmotsu manifold satisfies 
	\begin{gather}
	(\nabla_X Q)\xi=-QX-2nX,\label{Eq:11}\\
	(\nabla_\xi Q)X=-2QX-4nX,\label{Eq:12}
	\end{gather} 
	where $Q$ is the Ricci operator defined by $S(X,Y)=g(QX,Y)$.
\end{lemma}
\begin{proof}
	Note that \eqref{E:06} implies $Q\xi=-2n\xi$. Differentiation this, and recalling \eqref{E:04} provides \eqref{Eq:11}. 
	
	Now differentiating \eqref{E:05} along $W$ leads to
	\begin{equation*}
	(\nabla_W R)(X,Y)\xi=-R(X,Y)W+g(X,W)Y-g(Y,W)X.
	\end{equation*}
	Let $\{e_i \}_{i=1}^{2n+1}$ be a local orthonormal basis on $M$. Taking inner product of the above equation with $Z$ and then plugging $X=Y=e_i$ and summing over $i$ shows that
	\begin{equation}\label{Eq:13}
	g((\nabla_{e_i} R)(e_i,Y)\xi, Z)=S(Y,Z)+2ng(Y,Z).
	\end{equation}
	From second Bianchi identity, one can easily obtain	
	\begin{equation}\label{Eq:14}
	g((\nabla_{e_i}R)(Z,\xi)Y,e_i)=g((\nabla_Z Q)\xi,Y)-g((\nabla_\xi Q)Z,Y).
	\end{equation}
	Fetching \eqref{Eq:14} in \eqref{Eq:13} and using \eqref{Eq:11}, we obtain
	\begin{equation*}
	g((\nabla_\xi Q)Z,Y)=-2S(Y,Z)-4ng(Y,Z),
	\end{equation*}
	which proves \eqref{Eq:12}.
\end{proof}
Now we derive the expression of $*$-Ricci tensor on Kenmotsu manifolds.
\begin{lemma}
	On a Kenmotsu manifold, the $*$-Ricci tensor is given by
	\begin{equation}\label{Eq:15}
	S^*=S+(2n-1)g+\eta\otimes\eta.
	\end{equation}
\end{lemma}
\begin{proof}
	From the definition of *-Ricci tensor we also have
	\begin{equation*}
	S^*(X,Y)=\frac{1}{2}trace(Z\to \varphi R(X, \varphi Y)Z),
	\end{equation*}
	which by virtue of first Bianchi identity gives
	\begin{equation}\label{EQ:20}
	S^*(X,Y)=\frac{1}{2}trace\{Z\to -\varphi R(\varphi Y,Z)X\}+\frac{1}{2}trace\{Z\to -\varphi R(Z,X)\varphi Y\}.
	\end{equation}
	We now recall the following identities (see \cite{Kenm}):
	\begin{align}
	R(X,Y)\varphi Z-\varphi R(X,Y)Z=& g(Y,Z)\varphi X-g(X, Z)\varphi Y\nonumber\\	&+g(X,\varphi Z)Y-g(Y, \varphi Z)X,\label{E:3.6}\\
	R(\varphi X, \varphi Y)Z=R(X,Y)Z&+g(Y, Z)X-g(X,Z)Y\nonumber\\
	&+g(Y, \varphi Z)\varphi X-g(X, \varphi Z)\varphi Y.\label{E:3.7}
	\end{align}
	for any $X,Y, Z\in \mathfrak{X}(M)$. Taking inner product of \eqref{E:3.6} with $\varphi W$, and then making use of skew-symmetry of $\varphi$ and \eqref{E:02}, we obtain
	\begin{align*}
	g(R(X,Y)\varphi Z, \varphi W)-g(R(X,Y)Z,W)=g(Y,Z)g(X,W)-g(X,Z)g(Y,W)\\
	+g(X,\varphi Z)g(Y, \varphi W)-g(Y,\varphi Z)g(X, \varphi W).
	\end{align*} 
	Let $\{e_i\}_{i=1}^{2n+1}$ with $e_{2n+1}=\xi$ be a local orthonormal basis on $M$. Then putting $Y=Z=e_i$ in the above equation and summing over $i$ gives
	\begin{equation}\label{EQ:23}
	trace\{Z\to -\varphi R(Z,X)\varphi W\}=S(X,W)+(2n-1)g(X,W)+\eta(X)\eta(W).
	\end{equation}
	On the other hand it follows from \eqref{E:3.7} that
	\begin{equation}\label{EQ:24}
	\sum_{i=1}^{2n} g(R(\varphi Y, \varphi e_i)e_i, W)=S(Y,W)+(2n-1)g(Y,W)+\eta(Y)\eta(W),
	\end{equation}
	where we used \eqref{E:05}. Note that, if $\{e_i\}_{i=1}^{2n+1}$ with $e_{2n+1}=\xi$ is an orthonormal basis of vector fields on $M$, then $\{\varphi e_i, \xi\}_{i=1}^{2n}$ is also a local orthonormal basis  on $M$. Thus, as $\varphi \xi=0$, it follows that
	\begin{align}\label{EQ:25}
	\frac{1}{2}trace\{Z\to -\varphi R(\varphi Y,Z)X\}&=\frac{1}{2} \sum_{i=1}^{2n} -g(\varphi R(\varphi Y, \varphi e_i) X, \varphi e_i)\nonumber\\
	&=\frac{1}{2} \sum_{i=1}^{2n} g(R(\varphi Y, \varphi e_i) e_i, X).
	\end{align}
	Now use of equations \eqref{EQ:25}, \eqref{EQ:24} and \eqref{EQ:23} in \eqref{EQ:20} gives \eqref{Eq:15}.
\end{proof}

Note that due to the presence of some extra terms in the expression of $*$-Ricci tensor, the defining condition of the $*$-Ricci soliton is different from Ricci soliton.
\begin{theorem}
	If the metric of a Kenmotsu manifold is a $*$-Ricci soliton, then the soliton constant $\lambda=0$.
\end{theorem}
\begin{proof}
	Feeding the expression of $*$-Ricci tensor as given by \eqref{Eq:15} into the $*$-Ricci soliton equation \eqref{Eq:03}, it follows that
	\begin{equation}\label{E:10}
	(\pounds_V g)(X, Y)=-2S(X,Y)-2(2n-1+\lambda)g(X,Y)-2\eta(X)\eta(Y).
	\end{equation}
	Taking covariant derivative of \eqref{E:10} along an arbitrary vector field $Z$, we get
	\begin{align}\label{E:11}
	(\nabla_Z \pounds_V g)(X,Y)=&-2\{(\nabla_Z S)(X,Y)+\eta(X)g(Y,Z)\nonumber\\
	&+\eta(Y)g(X,Z)-2 \eta(X)\eta(Y)\eta(Z) \}.
	\end{align}	
	From Yano \cite{Yano}, we know the following well known commutation formula:
	\begin{align*}
	(\pounds_V \nabla_X g-\nabla_X \pounds_V g- \nabla_{[V,X]} g)(Y, Z)=\\
	-g((\pounds_V \nabla)(X, Y), Z)-g((\pounds_V \nabla)(X, Z), Y),
	\end{align*}
	for all $X,Y, Z\in \mathfrak{X}(M)$. Since $\nabla g=0$, the above equation gives
	\begin{equation}\label{E:12}
	(\nabla_X \pounds_V g)(Y, Z)=g((\pounds_V \nabla)(X, Y), Z)+g((\pounds_V \nabla)(X, Z), Y),
	\end{equation}
	for all $X,Y, Z\in \mathfrak{X}(M)$. As $\pounds_V \nabla$ is a symmetric, it follows from \eqref{E:12} that
	\begin{align}\label{E:13}
	&g((\pounds_V \nabla)(X, Y), Z)\nonumber\\
	&=\frac{1}{2}(\nabla_X \pounds_V g)(Y, Z)+\frac{1}{2}(\nabla_Y \pounds_V g)(Z, X)-\frac{1}{2}(\nabla_Z \pounds_V g)(X, Y).
	\end{align}
	Making use of \eqref{E:11} in \eqref{E:13} we have
	\begin{align*}\label{E:14}
	g((\pounds_V \nabla)(X,Y), Z)=&(\nabla_Z S)(X,Y)-(\nabla_X S)(Y,Z)-(\nabla_Y S)(Z,X)\\
	&-2\eta(Z)g(X,Y)+2\eta(X)\eta(Y)\eta(Z).
	\end{align*}
	Plugging $Y=\xi$ in the above equation and using \eqref{Eq:11} and \eqref{Eq:12}, we have
	\begin{equation*}
	(\pounds_V \nabla)(X,\xi)=2QX+4nX.
	\end{equation*}
	Differentiating the above equation along $Y$ and using \eqref{E:04}, we obtain
	\begin{equation*}
	(\nabla_Y \pounds_V \nabla)(X,\xi)=-(\pounds_V \nabla)(X,Y)+2\eta(Y)\{QX+2nX \}+2(\nabla_Y Q)X.
	\end{equation*}
	Feeding the above obtained expression into the following well known formula (see Yano \cite{Yano})
	\begin{equation}\label{E:026}
	(\pounds_V R)(X,Y)Z=(\nabla_X\pounds_V\nabla)(Y,Z)-(\nabla_Y\pounds_V\nabla)(X,Z),
	\end{equation}
	and using the symmetry of $\pounds_V \nabla$, we immediately obtain
	\begin{align}\label{Eq:22}
	(\pounds_V R)(X,Y)\xi=&2\eta(X)\{QY+2nY \}-2\eta(Y)\{QX+2n X \}\nonumber\\
	&+2\{(\nabla_X Q)Y-(\nabla_Y Q)X \}.
	\end{align}
	Substituting $Y=\xi$ in the above equation, we get
	\begin{equation}\label{Eq:23}
	(\pounds_V R)(X,\xi)\xi=0.
	\end{equation}
	Now taking the Lie-derivative of $R(X,\xi)\xi=-X+\eta(X)\xi$ along $V$ gives
	\begin{equation*}
	(\pounds_V R)(X,\xi)\xi-2n\eta(\pounds_V \xi)X+g(X,\pounds_V \xi)\xi=(\pounds_V \eta)(X)\xi,
	\end{equation*}
	which by virtue of \eqref{Eq:23} becomes
	\begin{equation}\label{Eq:24}
	(\pounds_V \eta)(X)\xi=-2n\eta(\pounds_V \xi)X+g(X,\pounds_V \xi)\xi.
	\end{equation}
	With the help of \eqref{E:06}, the equation \eqref{E:10} takes the form
	\begin{equation}\label{Eq:25}
	(\pounds_V g)(X,\xi)=-2\lambda\eta(X).
	\end{equation}
	Substituting $X=\xi$ in the above equation gives
	\begin{equation}\label{Eq:26}
	\eta(\pounds_V \xi)=\lambda.
	\end{equation}
	Now Lie-differentiating $\eta(X)=g(X,\xi)$ yields $(\pounds_V \eta)(X)=(\pounds_V g)(X,\xi)+g(X,\pounds_V \xi)$. Using this and \eqref{Eq:26} in \eqref{Eq:24} provides 
	$\lambda(nX-\eta(X)\xi)=0.$
	Tracing the previous  equation yield $\lambda=0$.
\end{proof}
\begin{lemma}
	If the metric of a Kenmotsu manifold is a $*$-Ricci soliton, then the Ricci tensor satisfy
	\begin{equation}\label{Eq:016}
	(\pounds_V S)(X,\xi)=-X(r)+\xi(r)\eta(X).
	\end{equation}
\end{lemma}
\begin{proof}
	Contracting the equation \eqref{Eq:22} with respect to $X$ and recalling the following well-known formulas 
	\begin{equation*}
	\text{div}Q=\frac{1}{2}\text{grad }r, \qquad \text{trace}\nabla Q=\text{grad }r,
	\end{equation*}
	we easily obtain
	\begin{equation}\label{Eq:27}
	(\pounds_V S)(Y,\xi)=-Y(r)-2\eta(Y)\{r+2n(2n+1) \}.
	\end{equation}
	Substituting $Y=\xi$, we have $(\pounds_V S)(\xi,\xi)=-\xi(r)-2\{r+2n(2n+1) \}$. On the other hand, contracting \eqref{Eq:23} gives $(\pounds_V S)(\xi,\xi)=0.$ Using this in the previous equation leads to 
	\begin{equation}\label{Eq:28}
	\xi(r)=-2(r+2n(2n+1)).
	\end{equation}
	Hence \eqref{Eq:28} and \eqref{Eq:27} gives \eqref{Eq:016}.
\end{proof}
\begin{lemma}
	The scalar curvature $r$ of an $\eta$-Einstein Kenmotsu manifold $M$ of dimension $>3$ satisfies
	\begin{equation}\label{Eq:29}
	Dr=\xi(r)\xi.
	\end{equation}
\end{lemma}
\begin{proof}
	Since $M$ is $\eta$-Einstein, from \eqref{E:010} we have
	\begin{equation}\label{Eq:30}
	QX=\left( \frac{r}{2n}+1\right)X-\left( \frac{ r}{2n}+2n+1\right)\eta(X)\xi.
	\end{equation}
	Using \eqref{Eq:30}, \eqref{E:06} and \eqref{E:05}, one can easily verify that
	\begin{align}\label{E:4.02}
	R(X, Y)\xi=&\frac{1}{2n-1}\{S(Y, \xi)X+\eta(Y)QX-S(X,
	\xi)Y-\eta(X)QY\}\nonumber\\
	&-\frac{r}{2n(2n-1)}\{\eta(Y)X-\eta(X)Y\},
	\end{align}
	which also gives
	\begin{align}\label{E:4.03}
	R(X, \xi)Y=&\frac{1}{2n-1}\{g(Y, Q\xi)X+\eta(Y)QX-g(QX,Y)\xi-g(X,Y)Q\xi\}\nonumber\\
	&-\frac{r}{2n(2n-1)}\{\eta(Y)X-g(X,Y)\xi\}.
	\end{align}
	Putting $Y=\xi$ in \eqref{E:4.02} and then differentiating it along $W$ and using \eqref{E:4.03}, we get
	\begin{align*}
	(\nabla_WR)(X,\xi)\xi=&\frac{1}{2n-1}\{g((\nabla_WQ)\xi,\xi)X+(\nabla_WQ)X-g((\nabla_WQ)X,\xi)\xi\\
	&-\eta(X)(\nabla_WQ)\xi\}-\frac{Wr}{2n(2n-1)}\{ X-\eta(X)\xi\}.
	\end{align*}
	Taking inner product of above equation with $Y$ and contracting with respect to $X$ and $W$ yields
	\begin{align}\label{E:4.04}
	\sum_{i=1}^{2n+1} g((\nabla_{e_i}R)(e_i,\xi)\xi,Y)=&\frac{1}{2n-1}\{g((\nabla_YQ)\xi-(\nabla_\xi Q)Y,\xi)\}\nonumber\\
	&+\frac{2n-2}{4n(2n-1)}\{\varepsilon Yr-\eta(Y)\xi(r)\},
	\end{align}
	where $\{e_i\}$ is a local orthonormal basis on $M$. From second Bianchi identity we easily obtain
	\begin{equation}\label{E:4.05}
	\sum_{i=1}^{2n+1} g((\nabla_{e_i} R)(Y,\xi)\xi,e_i)=g((\nabla_Y Q)\xi-(\nabla_\xi Q)Y,\xi).
	\end{equation} 
	Then from \eqref{E:4.04} and \eqref{E:4.05} and noting that $n>1$ we get
	\begin{equation*}
	g((\nabla_Y Q)\xi-(\nabla_\xi Q)Y,\xi)=\frac{1}{4n}\{ Yr-\eta(Y)\xi(r)\}.
	\end{equation*}
	Since $\nabla Q$ is symmetric, the above equation becomes
	\begin{equation}\label{E:4.06}
	g((\nabla_Y Q)\xi,\xi)-g((\nabla_\xi Q)\xi,Y)=\frac{1}{4n}\{ Yr-\eta(Y)\xi(r)\}.
	\end{equation}
	On the other hand from \eqref{Eq:11} and \eqref{Eq:12}, the left hand side of above equation vanishes. Then \eqref{E:4.06} leads to $Yr=\eta(Y)\xi(r)$ which gives \eqref{Eq:29}.
\end{proof}
\begin{theorem}\label{T:3.2}
	Let $(M,\varphi, \xi, \eta, g)$ be an $\eta$-Einstein Kenmotsu manifold of dimension $>3$. If $g$ is a $*$-Ricci soliton, then $M$ is Einstein.
\end{theorem}
\begin{proof}
	Making use of \eqref{Eq:29} in \eqref{Eq:016}, we have $(\pounds_V S)(X,\xi)=0$.
	Now, Lie-differentiating \eqref{E:06} along $V$, using \eqref{E:010}, \eqref{Eq:25}, $\lambda=0$ and $\eta(\pounds_V \xi)=0$, we obtain
	\begin{equation*}
	(r+2n(2n+1))\pounds_V \xi=0.
	\end{equation*}
	Suppose if $r=-2n(2n+1)$, then \eqref{E:010} shows that $M$ is Einstein. 
	
	So that we assume $r\neq -2n(2n+1)$ in some open set $\mathcal{O}$ of $M$. Therefore on $\mathcal{O}$, we have $\pounds_V \xi=0$, and so it follows from \eqref{E:04} that 
	\begin{equation}\label{E:25}
	\nabla_\xi V=V-\eta(V)\xi.
	\end{equation}
	Clearly, \eqref{Eq:25} shows that $(\pounds_V g)(X,\xi)=0$ for any $X\in TM$. This together with \eqref{E:25}, it follows that
	\begin{equation}\label{E:26}
	g(\nabla_X V,\xi)=-g(\nabla_\xi V, X)=-g(X,V)+\eta(X)\eta(V).
	\end{equation}
	From Duggal and Sharma \cite{Duggal}, we know the following formula
	\begin{equation*}
	(\pounds_V \nabla)(X,Y)=\nabla_X\nabla_Y-\nabla_{\nabla_X Y} V+R(V,X)Y.
	\end{equation*}
	Setting $Y=\xi$ in the above equation and by virtue of \eqref{E:04}, \eqref{E:05}, \eqref{E:25} and \eqref{E:26}, we have $\xi r=0$, which with the help of \eqref{Eq:28} shows that  $r=-2n(2n+1)$. This leads to a contradiction and completes the proof.
\end{proof}

Now we consider Kenmotsu 3-manifolds which admits $*$-Ricci solitons.
\begin{theorem}\label{T:3.3}
	If the metric of a Kenmotsu 3-manifold is a $*$-Ricci soliton, then $M$ is of constant negative curvature $-1$.
\end{theorem}
\begin{proof}
	It is well known that the Riemannian curvature tenor of $3$-dimensional Riemannian manifold is given by
	\begin{align}\label{E:299}
	R(X,Y)Z=&g(Y,Z)QX-g(X,Z)QY+g(QY,Z)X-g(QX,Z)Y\nonumber\\
	&-\frac{r}{2}\{g(Y,Z)X-g(X,Z)Y \}.
	\end{align}
	Putting $Y=Z=\xi$ in \eqref{E:299} and using \eqref{E:05} and \eqref{E:06} gives 
	\begin{equation} \label{Eq:40}
	QX=\left( \frac{r}{2}+1\right)X-\left( \frac{r}{2}+3\right)\eta(X)\xi.
	\end{equation}
	Making use of above equation in \eqref{Eq:22} gives	\begin{equation}\label{E:32}
	(\pounds_V R)(X,Y)\xi=X(r)\{Y-\eta(Y)\xi\}
	+Y(r)\{-X+\eta(X)\xi \}.
	\end{equation}
	Replacing $Y$ by $\xi$ in the above equation and comparing it with \eqref{Eq:23}, we obtain
	\begin{equation*}
	\xi(r)\{-X+\eta(X)\xi \}=0.
	\end{equation*}
	Contracting the above equation with respect to $X$ gives $\xi(r)=0$, and consequently it follows from \eqref{Eq:28} that  $r=-6$. Then from \eqref{Eq:40} we have $QX=-2X$, and substituting this in \eqref{E:299} shows that $M$ is of constant negative curvature $-1$. 
\end{proof}
\begin{remark}
	In \cite{Ghosh}, Ghosh proved that, if $(M,g)$ is a Kenmotsu manifold of dimension $3$ and if $g$ is a Ricci soliton, then $M$ is of constant negative curvature $-1$, and here we have the same conclusion when $g$ is a $*$-Ricci soliton. Note that our approach and technique to obtain the result is  different to that of Ghosh.
\end{remark}
According to Kenmotsu \cite{Kenm}, the warped product space $\mathbb{R}\times_f K$, where $f(t)=ce^t$ on the real line $R$ and $K$ is K\"{a}hler manifold, admits a Kenmotsu structure. Consequently, we have the following result.
\begin{corol}
	Consider the warped product $M=\mathbb{R}\times_f K$, where $f(t)=ce^t$ and $K$ is a 2-dimensional K\"{a}hler manifold. If $M$ admits a $*$-Ricci soliton, then $M$ locally a hyperbolic space $\mathbb{H}^3(-1)$.
\end{corol}
We know that Kenmotsu manifold $M$ do not admit a Ricci soliton whose soliton vector field is equal to the characteristic vector field. Because, if $V=\xi$, then from \eqref{E:007} the Ricci soliton equation \eqref{E:00} would become
\begin{equation}\label{EQ:49}
S=-(1+\lambda)g+\eta\otimes\eta,
\end{equation}
which means $M$ is $\eta$-Einstein. But according to Ghosh (see Theorem~1 of \cite{Ghosh2} and Theorem~1 of \cite{Ghosh}) $M$ must Einstein, and this will be a contradiction to equation \eqref{EQ:49}. Thus in our study an interesting case arises when the soliton vector field $V$ of the *-Ricci soliton is equal to $\xi$, moreover, when $V$ is pointwise collinear with $\xi$. We consider this here and we prove the following.
\begin{theorem}
	If the metric of a Kenmotsu manifold is a *-Ricci soliton with a non-zero potential vector field $V$ pointwise collinear with the characteristic vector field $\xi$, then $M$ is Einstein and $V=\xi$.
\end{theorem}
\begin{proof}
	Observing the equation \eqref{E:10}, keeping in mind that $\lambda=0$ gives
	\begin{equation*}
	g(\nabla_X V,Y)+g(X,\nabla_Y V)=-2S(X,Y)-2(2n-1)g(X,Y)-2\eta(X)\eta(Y).
	\end{equation*}
	Putting $V=a\xi$, $a$ being a smooth function on $M$, in the above equation we have
	\begin{align*}
	da(X)\eta(Y)&+da(Y)\eta(X)+g(\nabla_X\xi, Y)+g(X,\nabla_Y \xi)\\
	&=-2S(X,Y)-2(2n-1+\lambda)g(X,Y)-2\eta(X)\eta(Y).
	\end{align*}
	Substituting $Y=\xi$ and using \eqref{E:04} and \eqref{E:06}, the above equation becomes 
	\begin{equation}\label{E:3.1}
	da(X)+da(\xi)\eta(X)=0.
	\end{equation}
	This means $a$ is invariant along the distribution $Ker \eta$, that is, $X(a)=0$ for all $X\in Ker \eta$.
	Replacing $X$ by $\xi$ in the above equation, we have
	\[
	da(\xi)=0.
	\]
	Thus $a$ is constant, and consequently we have
	\begin{equation*}
	a\pounds_\xi g=-2S-2(2n-1)g-2\eta\otimes\eta.
	\end{equation*}
	Making use of \eqref{E:007} in the above equation gives
	\begin{equation}\label{Eq:44}
	S=(1-a-2n)g+(a-1)\eta\otimes\eta,
	\end{equation}
	which means $M$ is $\eta$-Einstein. Then it follows from Theorem~\ref{T:3.2} and Theorem~\ref{T:3.3} that $M$ is Einstein. Hence \eqref{Eq:44} gives $a=1$ and the result follows.
\end{proof}
Now we provide an example of a Kenmotsu 3-manifold which admits a $*$-Ricci soliton and verify our results.
\begin{example}\label{Ex:01}
	Let $M=N \times I$, where $N$ is an open connected subset of $\mathbb{R}^2$ and $I$ is an open interval in $\mathbb{R}$.
	Let $(x,y,z)$ be the Cartesian coordinates in it. Define the structure $(\varphi, \xi, \eta, g)$ on $M$ as follows:
	\begin{align*}
	\varphi\left( \frac{\partial}{\partial x}\right)& = \frac{\partial}{\partial y},\quad \varphi\left( \frac{\partial}{\partial y}\right) =-\frac{\partial}{\partial x},\quad \varphi\left( \frac{\partial}{\partial z}\right) =0, \\
	\xi&= \frac{\partial}{\partial z},\quad \eta=dz,\\
	(g_{ij})&=
	\begin{pmatrix}
	e^{2z}\quad & 0\quad & 0\\
	0\quad & e^{2z}\quad & 0\\
	0\quad & 0\quad & 1
	\end{pmatrix}.
	\end{align*}
	Now from Koszul's formula, Levi-Civita connection $\nabla$ is given by
	\begin{equation}\label{E:43}
	\begin{aligned}
	\nabla_{\partial_x}\partial_x &=-e^{2z}\partial_z, &\nabla_{\partial_x}\partial_y&=0, &\nabla_{\partial_x}\partial_z&=\partial_x,\\
	\nabla_{\partial_y}\partial_x &=0, &\nabla_{\partial_y}\partial_y&=-e^{2z}\partial_z, &\nabla_{\partial_y}\partial_z&=\partial_y,\\
	\nabla_{\partial_z}\partial_x &=\partial_x, &\nabla_{\partial_z}\partial_y&=\partial_y, &\nabla_{\partial_z}\partial_z&=0.
	\end{aligned}
	\end{equation}
	where $\partial_x=\frac{\partial}{\partial x}, \partial_y=\frac{\partial}{\partial y}$ and $\partial_z=\frac{\partial}{\partial z}$. From \eqref{E:43}, one can easily verify \eqref{EQ:06}, and so $(M, \varphi, \xi, \eta, g)$ is a Kenmotsu manifold.
	
	With the help of \eqref{E:43}, we find the following: 
	\begin{equation}\label{E:44}
	\begin{gathered}
	R(\partial_x, \partial_y)\partial_z=R(\partial_y, \partial_z)\partial_x =R(\partial_x, \partial_z)\partial_y =0,\\
	R(\partial_x, \partial_z)\partial_x=R(\partial_y, \partial_z)\partial_y=e^{2z}\partial_z,\\
	R(\partial_x, \partial_y)\partial_x=e^{2z}\partial_y, \quad R(\partial_y, \partial_z)\partial_z=-\partial_y,\\
	R(\partial_x, \partial_z)\partial_z=-\partial_x, \quad R(\partial_x, \partial_y)\partial_y=-e^{2z}\partial_x.
	\end{gathered}
	\end{equation}
	Let $e_1=e^{-z}\partial_x, e_2=e^{-z}\partial_y$ and $e_3=\xi=\partial_z$. Clearly, $\{e_1, e_2,e_3\}$ forms an orthonormal $\varphi$-basis of vector fields on $M$. Making use of \eqref{E:44} one can easily show that $M$ is Einstein, that is, $S(X,Y)=-2ng(X,Y)$, for any $X,Y\in \mathfrak{X}(M)$. Also from the definition of *-Ricci tensor it is not hard to see that 
	\begin{equation}\label{Eq:59}
	S^*(X,Y)=-g(X,Y)+\eta(X)\eta(Y),
	\end{equation}
	for any  $X,Y\in \mathfrak{X}(M)$.
	
	Let us consider the vector field
	\begin{equation}\label{E:45}
	V=(1-a)x\partial x+(1-a)y\partial y+a\partial z,
	\end{equation}
	where $a\neq 1$ is a constant. Using \eqref{E:43} one can easily verify that 
	\begin{equation}\label{Eq:61}
	(\pounds_V g)(X,Y)=2\{g(X,Y)-\eta(X)\eta(Y) \},
	\end{equation}
	for any  $X,Y\in \mathfrak{X}(M)$. Combining \eqref{Eq:61} and \eqref{Eq:59}, we obtain that $g$ is a $*$-Ricci soliton, that is, \eqref{Eq:03} holds true with $V$ as in \eqref{E:45} and $\lambda=0$. Further \eqref{E:44} shows that $R(X,Y)Z=-\{g(Y,Z)X-g(X,Z)Y\}$ for any  $X,Y, Z\in \mathfrak{X}(M)$, which means $M$ is of constant negative curvature $-1$ and this verifies the Theorem~\ref{T:3.3}.
\end{example}

\section{Gradient almost $*$-Ricci solitons on Kenmotsu manifolds}
First, we prove the following result.
\begin{lemma}
	If the metric of a Kenmotsu manifold is a gradient almost $*$-Ricci soliton, then the Riemannian curvature tensor $R$ can be expressed as
	\begin{align}\label{Eq:42}
	R(X,Y)Df=(\nabla_Y Q)X-(\nabla_XQ)Y+Y(\lambda)X-X(\lambda)Y+\eta(X)Y-\eta(Y)X,
	\end{align}
	for any $X,Y\in \mathfrak{X}(M)$.
\end{lemma}
\begin{proof}
	Using the expression of $*$-Ricci tensor as given in \eqref{Eq:15} into the definition of gradient almost *-Ricci soliton, we get
	\begin{equation*}
	\nabla_X Df=-QX-(\lambda+2n-1)X-\eta(X)\xi.
	\end{equation*}
	Differentiating the above equation along an arbitrary vector field $Y$, we have
	\begin{align*}
	\nabla_Y\nabla_X Df=&-(\nabla_Y Q)X-Q(\nabla_Y X)-Y(\lambda)X-(\lambda+2n-1)\nabla_Y X\\
	&-g(X,Y)\xi+\eta(X)\eta(Y)\xi-\eta(\nabla_Y X)\xi-\eta(X)Y+\eta(X)\eta(Y)\xi.
	\end{align*}
	Then applying the above equation in the expression of Riemannian curvature tensor, we obtain the desired result.
\end{proof}
Now we prove the following fruitful result.
\begin{theorem}\label{T:4.1}
	If the metric of a Kenmotsu manifold $M$ is a gradient almost $*$-Ricci soliton, then either $M$ is Einstein or the soliton vector field $V$ is pointwise collinear with the characteristic vector field $\xi$ on an open set $\mathcal{O}$ of $M$.
\end{theorem}
\begin{proof}
	Taking inner product of \eqref{Eq:42} with $\xi$, with the help of \eqref{Eq:11}, we get
	\begin{equation*}
	g(R(X,Y)Df,\xi)=Y(\lambda)\eta(X)-X(\lambda)\eta(Y).
	\end{equation*}
	Now taking inner product of \eqref{E:05} with $Df$ gives
	\begin{equation*}
	g(R(X,Y)\xi,Df)=Y(f)\eta(X)-X(f)\eta(Y).
	\end{equation*}
	Comparing the last two equations and then plugging $Y=\xi$ yields $X(f+\lambda)=\xi(f+\lambda)\eta(X)$, from which we have
	\begin{equation}\label{Eq:45}
	d(f+\lambda)=\xi(f+\lambda)\eta,
	\end{equation}
	which means $f+\lambda$ is invariant along the distribution $Ker \eta$, that is, $X(f+\lambda)=0$ for all $X\in Ker \eta$.
	
	Now putting $X=\xi$ in \eqref{Eq:42} and then taking inner product with $Z$ gives
	\begin{equation*}
	g(R(\xi, Y)Df, Z)=S(X,Z)+(2n+1-\xi(\lambda)) g(X,Z)+Z(\lambda)\eta(X)-\eta(X)\eta(Z).
	\end{equation*}
	On the other hand, from \eqref{E:05} we have
	\begin{equation*}
	g(R(\xi,Z)X,Df)=-g(X,Z)\xi(f)+\eta(X)Z(f).
	\end{equation*}
	Combining the above two equations, one can easily obtain
	\begin{equation}\label{Eq:48}
	S(X,Z)=(\xi(f+\lambda)-2n-1)g(X,Z)+(1-\xi(f+\lambda))\eta(X)\eta(Z),
	\end{equation}
	which means $M$ is $\eta$-Einstein. Contracting the above equation, one immediately obtain
	\begin{equation}\label{Eq:49}
	\xi(f+\lambda)=\frac{r}{2n}+(2n+2).
	\end{equation}
	Using the above equation in \eqref{Eq:48} one can easily have \eqref{E:010}.  Now contracting \eqref{Eq:42} over $X$, we obtain
	\begin{equation*}
	S(Y,Df)=\frac{1}{2}Y(r)+2n Y(\lambda)-2n\eta(Y).
	\end{equation*}
	Comparing the above equation with \eqref{E:010} shows that
	\begin{equation}\label{Eq:51}
	(r+2n)Y(f)-(r+2n(2n+1))\eta(Y)\xi(f)-nY(r)-4n^2 Y(\lambda)+4n^2 \eta(Y)=0,
	\end{equation}
	for all $Y\in \mathfrak{X}(M)$. Substituting $Y=\xi$ in the above equation, it follows from \eqref{Eq:49} that 
	\begin{equation}\label{EQ:61}
	\xi(r)=-2(r+2n(2n+1)).
	\end{equation}
	Operating \eqref{Eq:45} by $d$, and since $d^2=0$ and $d\eta=0$, we obtain $dr\wedge\eta=0$, that is,
	\begin{equation*}
	dr(X)\eta(Y)-dr(Y)\eta(X)=0,
	\end{equation*} 
	for all $X,Y\in \mathfrak{X}(M)$. After replacing $Y$ by $\xi$ in the above equation and using \eqref{EQ:61}, we have $X(r)=-2(r+2n(2n+1))\eta(X)$, which means
	\begin{equation}\label{Eq:54}
	Dr=-2(r+2n(2n+1))\xi.
	\end{equation}
	Let $X\in Ker \eta$ be arbitrary. Then with the help of \eqref{Eq:54}, the equation \eqref{Eq:51} becomes
	\begin{equation*}
	(r+2n)X(f)-4n^2X(\lambda)=0.
	\end{equation*}
	Using \eqref{Eq:45} and \eqref{Eq:49} in the above equation, we have $(r+2n(2n+1))X(f)=0$. This implies that 
	\begin{equation*}
	(r+2n(2n+1))(Df-\xi(f))=0.
	\end{equation*}
	If $r=-2n(2n+1)$, then it follows from the $\eta$-Einstein condition \eqref{E:010} that $S=-2ng$, and hence $M$ is Einstein. Suppose if $r\neq-2n(2n+1)$ on some open set $\mathcal{O}$ of $M$, then we have $Df=\xi(f)$, and this completes the proof.
\end{proof}
\begin{corol}\label{C:4.1}
	If the metric of a Kenmotsu 3-manifold is a gradient almost $*$-Ricci soliton, then either $M$ is of constant curvature $-1$ or the soliton vector field $V$ is pointwise collinear with the characteristic vector field $\xi$ on an open set $\mathcal{O}$ of $M$.
\end{corol}
\begin{proof}
	From Theorem~\ref{T:4.1} we have $M$ is Einstein, that is, $QX=-2nX$. Using this in \eqref{E:299} shows that $M$ is of constant negative curvature $-1$. 
\end{proof}

Now we provide an example of a gradient almost $*$-Ricci soliton on a Kenmotsu 3-manifold. 
\begin{example}
	Let $M=N \times I$, where $N$ is an open connected subset of $\mathbb{R}^2$ and $I$ is an open interval in $\mathbb{R}$.
	Let $(x,y,z)$ be the Cartesian coordinates in it. Define the Kenmotsu structure $(\varphi, \xi, \eta, g)$ on $M$ as in Example~\ref{Ex:01}.
	
	Let $f:M\to \mathbb{R}$ be a smooth function defined by 
	\begin{equation}\label{EQ:63}
	f(x,y,z)=-xe^z+z.
	\end{equation}
	Then the gradient of $f$ with respect to the metric $g$ is given by 
	\begin{equation*}
	Df=-e^{-z}\partial_x+(1-xe^z)\partial_z.
	\end{equation*}
	With the help of \eqref{E:43} one can easily verify that
	\begin{equation*}
	(\pounds_{Df} g)(X,Y)=2\{(1-xe^z)g(X,Y)-\eta(X)\eta(Y) \}.
	\end{equation*}
	Using \eqref{Eq:59} in the above equation, we have
	\begin{equation*}
	(\pounds_{Df} g)(X,Y)+2S^*(X,Y)+2xe^z g(X,Y)=0,
	\end{equation*}
	for all $X,Y\in \mathfrak{X}(M)$. Thus $g$ is a gradient almost $*$-Ricci soliton with the soliton vector field $V=Df$, $f$ as given in \eqref{EQ:63} and $\lambda=xe^z$. As shown in Example~\ref{Ex:01}, $M$ is Einstein and is of constant negative curvature $-1$. This verifies Theorem~\ref{T:4.1} and Corollary~\ref{C:4.1} 
\end{example}

\end{document}